\documentclass[12pt,reqno]{amsart}

\usepackage[x11names]{xcolor}
\usepackage[linkcolor=Blue2,citecolor=red,colorlinks]{hyperref}
\usepackage[capitalize,nameinlink]{cleveref}
\usepackage{color, bm, amscd, tikz-cd}

\newtheorem{thm}{Theorem}[section]

\newtheorem{corollary}[thm]{Corollary}
\newtheorem{lemma}[thm]{Lemma}

\theoremstyle{definition}

\numberwithin{equation}{section}

\begin{document}
	\title[Tetrablock]{A family of $2 \times 2$ tetrablock contractions with non-commuting fundamental operators}
	\author[M. Bhowmik]{Mainak Bhowmik}
	\address{Department of Mathematics\\
	Indian Institute of Technology Bombay\\ 
	 Powai, Mumbai, 400076, India}
	\email{mainak.bhowmik943@gmail.com, mainak@math.iitb.ac.in}

	\thanks{2020 {\em Mathematics Subject Classification}: 47A20, 47A13.\\
		{\em Keywords and phrases}:  Tetrablock, rational dilation, fundamental operators, complete spectral property}

	\maketitle
	\begin{abstract}
		In this note, we provide a family of $2\times 2$ tetrablock contractions that have tetrablock isometric dilation, but the corresponding fundamental operators do not commute. This answers a question raised by Bhattacharyya [Indiana Univ. Math. J. 2014] about the necessity of commuting fundamental operators for a tetrablock contraction to have rational dilation in a negative direction.
	\end{abstract}

	\section{Introduction \& Background}
	Let $K$ be a compact set in the $d$ dimensional complex Euclidean space $\mathbb{C}^d$. Let $\mathcal{R}(K)$ be the algebra of rational functions in $d$ variables with poles off $K$, endowed with the supremum norm given by $\|f\|_{\infty, K} = \{ |f(\boldsymbol{z})|: \boldsymbol{z} \in K \}$.
	
	We say that $K$ is a {\em spectral set} for a $d$-tuple $\boldsymbol{T} = \left(T_1,T_2,\dots,T_d \right)$ of commuting bounded operators on a Hilbert space $\mathcal{H}$, if the Taylor joint spectrum, $\sigma(\boldsymbol{T})\subseteq K $ and 
	\begin{align}   \label{Eq:Spectral}
	 \| f(\boldsymbol{T}) \| \leq \| f \|_{\infty, K} \ \ \text{for every } f \in \mathcal{R}(K),
	 \end{align}
	  where $f(\boldsymbol{T})$ is formed according to Taylor's functional calculus. In other words, the map, $\rho_{\boldsymbol{T}}: f \mapsto f(\boldsymbol{T})$ is a unital contractive algebra homomorphism from $\mathcal{R}(K)$ to $\mathbb{C}$.
	
	The set $K$ is said to be a \textit{complete spectral set} for $\boldsymbol{T}$ if for each $n\in \mathbb{N}$ and any $n\times n$ matrix-valued rational function $f= \left(\left( f_{jk}\right)\right)$ in $M_n\left(\mathcal{R}(K)\right) $,  we have 
	$$ \|f(\boldsymbol{T})\| \leq \| f \|_{\infty, K}=\sup\left\lbrace \| \left((f_{jk}(\boldsymbol{z}))\right)\|_{op} : \boldsymbol{z} \in K \right\rbrace, $$
	where $ f(\boldsymbol{T})= \left(\left( f_{jk}(\boldsymbol{T}) \right)\right) $ is a bounded operator from the direct sum of $n$ copies of $\mathcal{H}$ to itself i.e., the map
	$\operatorname{Id}_n \otimes \rho_{\boldsymbol{T}} :  f \mapsto f(\boldsymbol{T})$ from $M_n(\mathbb{C}) \otimes \mathcal{R}(K) $ to $M_n(\mathbb{C}) \otimes \mathcal{B}(\mathcal{H})$ is a contractive algebra homomorphism for every $n$, where $\mathcal{B}(\mathcal{H})$ is the $C^*$-algebra of bounded linear operators on $\mathcal{H}$. In such a case, we say that $\rho_{\boldsymbol{T}}$ is {\em completely contractive}.
	
	A $\partial K $-\textit{normal dilation} for a tuple $\boldsymbol{T}$ with $\sigma(\boldsymbol{T}) \subseteq K $ is a tuple of commuting bounded normal operators $\boldsymbol{N} = \left(N_1,N_2,\dots,N_d \right)$ on a Hilbert space $\mathcal{E}$ containing $\mathcal{H}$ as a closed subspace satisfying 
	\begin{align}\label{Eq:Dilation}
			f(\boldsymbol{T}) = \text{P}_{\mathcal{H}} f(\boldsymbol{N}) \vert_{\mathcal{H}},
	\end{align}
		for every $f \in \mathcal{R}(X)$ and $\sigma(\boldsymbol{N})\subseteq \partial K$, the Silov boundary of $K$ with respect to the uniform algebra $\mathcal{A}(K)$ which is the uniform closure in $C(K)$ of the functions holomorphic on a neighbourhood of $K$. Here, $\text{P}_\mathcal{H}$ is the orthogonal projection from $\mathcal{K}$ onto $\mathcal{H}$.
		
		 We say that rational dilation \textit{holds} on $K$ if $K$ is a complete spectral set for $\boldsymbol{T}$ whenever $K$ is a spectral set for $\boldsymbol{T}$. Else, rational dilation \textit{fails}.
		

	The above notions and terminologies arose after the pioneering works of von Neumann and Sz.-Nagy. In this modern language, von Neumann's inequality states that {\em every Hilbert space contraction has the open unit disk $\mathbb{D}$ as a spectral set}, and the Sz.-Nagy dilation theorem states that {\em every Hilbert space contraction $T$ has a $\partial \mathbb{D}$-normal dilation}. The following celebrated result of Arveson connects the notions of rational dilation with the complete spectral set.

	\begin{thm}[Arveson \cite{Arveson}]
		Let $\boldsymbol{T}$ have $K$ as a spectral set. Then $\boldsymbol{T}$ has $K$ as a complete spectral set if and only if $\boldsymbol{T}$ has a $\partial K$-normal dilation.
	\end{thm}
	
	Like the open unit disk, the rational dilation also holds on the bidisk, the annulus, and the symmetrized bidisk. It fails on the polydisk $\overline{\mathbb{D}}^d$ for $d\ge 3$, on planar domains with two or more holes, and on any norm unit ball in $\mathbb{C}^d$ for $d\geq 3$. The search for a domain in dimension higher than two where rational dilation holds is still on.
	
		In a landmark work \cite{Agler-Invent}, Agler established a nice connection between complex geometry and operator theory. More precisely, he used the Carath\'eodory distance between two points $\boldsymbol{z}$ and $\boldsymbol{w}$ in a bounded domain $\Omega$ in $\mathbb{C}^d$ to compute the largest angle between the eigenfunctions of a commuting $d$-tuple $\boldsymbol{T}$ of $2\times 2$ matrices having joint spectrum in $\{\boldsymbol{z}, \boldsymbol{w}\}$ and satisfying \eqref{Eq:Spectral} with $K=\Omega$ and $f$ being holomorphic functions in $\Omega$. This led him to prove the following result related to the complete spectral set criterion. It is crucial for our present note.
		
		\begin{thm}[Agler \cite{Agler-Invent}] \label{Thm:Agler}
		Let $K$ be a compact pseudoconvex subset of $\mathbb{C}^d$.  A unital contractive algebra homomorphism $\Pi: \mathcal{A}(K) \rightarrow M_2(\mathbb{C})$ is completely contractive. 
	\end{thm}
	
	In the backdrop of this, we consider the {tetrablock}, a bounded domain that arose in the study of $\mu$-synthesis problem, defined as
	$$\mathbb{E} = \left \lbrace (a_{11}, a_{22}, \det A) \in \mathbb{C}^3 :  A= \left(\begin{array}{cc}
		a_{11} &  a_{12} \\
		a_{21} & a_{22} \\
	\end{array}\right) \in M_2(\mathbb{C}) \text{ with } \|A\|_{op} < 1 \right \rbrace \subset \mathbb{C}^3.
	$$ 
It is a bounded $(1, 0, 1)$-balanced and $(0, 1, 1)$-balanced pseudoconvex domain \cite{EKZ}. If we consider $\|A\|_{op} \leq 1$ in the above description, we obtain the closed tetrablock $\overline{\mathbb{E}}$, which is a polynomially convex compact set and hence, it is a pseudoconvex compact set (or, Stein compactum). There has been a fascinating history of the theory of commuting triples $(T_1, T_2, T_3)$ of operators with $\overline{\mathbb{E}}$ as a spectral set. Such a triple is called a \textit{tetrablock contraction} or an {\em $\overline{\mathbb{E}}$-contraction}. In order to show a triple to be a $\overline{\mathbb{E}}$-contraction, it is enough to check the \eqref{Eq:Spectral} for $f$ in the polynomial algebra $\mathbb{C}[z_1, z_2, z_3]$ instead of the rational algebra $\mathcal{R}(\overline{\mathbb{E}})$ as the closed tetrablock is polynomially convex.
A special class of tetrablock contractions was dilated in \cite{Bhattacharyya-IUMJ}.
	{\em The rational dilation problem on the tetrablock is still unresolved \cite{BS-JFA}.}
	
	In the context of the tetrablock, a boundary normal dilation consists of a triple $(N_1, N_2, N_3)$ of commuting normal operators with joint spectrum contained in the distinguished boundary
	$$ b\mathbb E =\left\lbrace (z_1, z_2, z_3)\in \overline{\mathbb{E}} : |z_3|=1 \right\rbrace.$$
	Such a normal triple is called a {\em tetrablock unitary}. A  {\em tetrablock isometry} is the restriction of a tetrablock unitary to a joint invariant subspace. Although the notion of rational dilation demands that given a tetrablock contraction $(T_1, T_2, T_3)$, one must have a normal triple as above, it is not hard to see that finding a Hilbert space $\mathcal E$ containing $\mathcal H$ and a tetrablock isometry $(V_1, V_2, V_3)$ on $\mathcal E$ such that
	$$V_1^*|_{\mathcal H} = T_1^*, V_2^*|_{\mathcal H} = T_2^* \text{ and } V_3^*|_{\mathcal H} = T_3^*,$$
	is enough.
	For more details, please refer to \cite{Bhattacharyya-IUMJ}. 
	Thus, a direct consequence of Theorem \ref{Thm:Agler} is the following.
	\begin{corollary} \label{Cor:Main}
		Every $2 \times 2$ tetrablock contraction has a tetrablock unitary dilation.
	\end{corollary}
	
	\newpage
	
	Now we shall define the key protagonists of this note.
	\subsection*{The Fundamental operators} Let $(T_1, T_2, T_3)$ be a tetrablock contraction on a Hilbert space $\mathcal{H}$. It has been shown in \cite[Theorem 3.4] {Bhattacharyya-IUMJ} that there exist unique bounded linear operators $F_1$ and $F_2$ on $\mathcal{D}_{T_3} = \overline{\operatorname{Ran}} D_{T_3}$ such that 
	\begin{align}
		T_1-T_2^*T_3 =D_{T_3} F_1 D_{T_3} \ \ \text{and} \ \ T_2-T_1^*T_3 = D_{T_3} F_2 D_{T_3},
	\end{align}
	where $D_{T_3} = \left(I- T_3^*T_3\right)^{1/2}$. These two operators are called the {\em fundamental operators} for the tetrablock contraction $(T_1, T_2, T_3)$. 
	
	\begin{thm}[Bhattacharyya \cite{Bhattacharyya-IUMJ}] \label{Thm: B}
		Let $(T_1, T_2, T_3)$ be a tetrablock contraction on a Hilbert space $\mathcal{H}$ with the fundamental operators $F_1$ and $F_2$. Define three bounded linear operators $V_1, V_2,$ and $V_3$ on $\mathcal{K}= \mathcal{H} \oplus \ell^2(\mathcal{D}_{T_3})$ as follows.
		\begin{align}\label{Eq:Dil-tuple}
			V_j= \begin{bmatrix}
				T_j & 0 \\ C_j & T_{\varphi_j}
			\end{bmatrix}: \mathcal{H} \oplus \ell^2(\mathcal{D}_{T_3}) \rightarrow  \mathcal{H} \oplus \ell^2(\mathcal{D}_{T_3}) , \ \text{for } j=1, 2, 3,
				\end{align}
				where $T_{\varphi_1}, T_{\varphi_2}$, and $T_{\varphi_3}$ are the Topelitz operators on $\ell^2(\mathcal{D}_{T_3})$ associated with the symbols $\varphi_1(z)=F_1 + F_2^*z $, $\varphi_2(z)=F_2 + F_1^*z$, and $\varphi_3(z)=I_{\mathcal{D}_{T_3}}z$, respectively. Here, $C_1= (F_2^*D_{T_3}, 0, \dots)^t$, $C_2= (F_1^*D_{T_3}, 0, \dots)^t$, and $C_3=(0,  D_{T_3}, \dots)^t$  are the column operators. Then we have the following:
		\begin{enumerate}
			\item The triple $(V_1, V_2, V_3)$ is a tetrablock isometric dilation of $(T_1, T_2, T_3)$ if $[F_1, F_2]=0$ and $[F_1, F_1^*] = [F_2, F_2^*]$.
			\item If there is a tetrablock isometric dilation $(W_1, W_2, W_3)$ of $(T_1, T_2, T_3)$ such that $W_3$ is the minimal isometric dilation of $T_3$, then $(W_1, W_2, W_3)$ is unitarily equivalent to $(V_1, V_2, V_3)$. Moreover, $[F_1, F_2]=0$ and $[F_1, F_1^*] = [F_2, F_2^*]$.
		\end{enumerate}
	\end{thm}
	
	This constructive dilation of the tetrablock contraction $(T_1, T_2, T_3)$ on {\em the minimal isometric dilation space of $T_3$} under the above conditions on the fundamental operators motivates Bhattacharyya to pose the following question.
	\vspace{2mm}
	
	\textbf{Question.} {\em Are the conditions $[F_1, F_2]=0$ and $[F_1, F_1^*] = [F_2, F_2^*]$ necessary to have tetrablock isometric dilation of $(T_1, T_2, T_3)$?}
	
	\vspace{2mm}
	
It has been shown by explicit examples in \cite[Example 3.4]{BS-JFA} and \cite[Example 2.1]{BB} that the condition $[F_1, F_1^*] = [F_2, F_2^*]$ is not necessary to have tetrablock isometric dilation of a tetrablock contraction $(T_1, T_2, T_3)$. In fact, an explicit tetrablock isometric dilation was found in \cite[Section 2]{BB} such that the dilation space is not the minimal isometric dilation of $T_3$.  Thus, it remains to investigate whether the commutativity of the fundamental operators is necessary.
In the next section, we provide a family of $2 \times 2$ tetrablock contractions that have tetrablock unitary dilations by Corollary \ref{Cor:Main}, but the corresponding fundamental operators do not commute. En route, we show how the Schwarz-type lemma for the polydisk can be used in proving a commuting triple of $2 \times 2$ upper triangular matrices to be a tetrablock contraction - a difficult task, in general.

	\newpage

\section{Main construction}
	We start with a commuting triple, $(T_1, T_2, T_3)$ of $2 \times 2$ upper triangular matrices as follows:
	\begin{align*}
		T_j = \begin{bmatrix}
			0 & \lambda_j \\ 0 & 0
			\end{bmatrix} \ \text{ for } j=1, 2, 3.
	\end{align*}

	The following lemmas are crucial for producing the required tetrablock contractions.
	\begin{lemma} \label{Lem: Mainak}
		There exists $r \in (0, 1)$ such that the polydisk 
		$r \overline{\mathbb{D}^3}= \{(rz_1, rz_2, rz_3): (z_1, z_2, z_3) \in \mathbb{D}^3 \}$ is contained in $ \mathbb{E}$.
	\end{lemma}
	\begin{proof}
		Since $(0, 0, 0)$ is an interior point of $\mathbb{E}$, we can find $0<r<1$ such that the polydisk 	$r \overline{\mathbb{D}^3} \subset \mathbb{E}$.
	\end{proof}
	
We state the following lemma due to Knese. 
	\begin{lemma}[Knese \cite{Knese-PAMS}] \label{Lem: Knese}
		If $g: \mathbb{D}^d \rightarrow \overline{\mathbb{D}}$ is holomorphic, then 
		\begin{align*}
			\sum_{j=1}^d (1-|z_j|^2) \left| \frac{\partial g}{\partial z_j}(z_1, \dots, z_d) \right| \leq 1- |g(z_1, \dots, z_d)|^2.
		\end{align*}
	\end{lemma}
	Let $r$ be as in Lemma \ref{Lem: Mainak} from now on. The following theorem provides us with the aimed family of tetrablock contractions.
	\begin{thm}\label{Thm: Mainak}
The triple $(T_1, T_2, T_3)$ is a tetrablock contraction having non-commuting fundamental operators whenever $0< | \lambda_j | < r$ for $j=1, 2, 3$ and $|\lambda_1| \neq |\lambda_2|$.
	\end{thm}
	\begin{proof}
			Let us recall that the Taylor joint spectrum for a tuple of commuting matrices are their joint eigenvalues. Also, $T_j$ is a nilpotent matrix and hence, $0$ is the only eigenvalue for each $j \in \{1, 2, 3\}$. Therefore, $(0, 0, 0)$ is the only joint eigenvalue with $(1, 0) \in \mathbb{C}^2$ as a joint eigenvector for the commuting triple $(T_1, T_2, T_3)$. Thus, the Taylor joint spectrum $\sigma(T_1, T_2, T_3) = \{ (0, 0 , 0)\} \subseteq \overline{\mathbb{E}}$. 
			
		Now, let $f$ be a function in $\mathcal{A}(\overline{\mathbb{E}})$ such that $\| f\|_{\infty, \overline{\mathbb{E}}} \leq 1$ and  $f(0, 0, 0)=0$.
		Now,
		\begin{align*}
			f(T_1, T_2, T_3) = \begin{bmatrix}
				0 & \sum_{j=1}^3 \lambda_j \frac{\partial f}{\partial z_j}(0, 0, 0) \\
				0 & 0
			\end{bmatrix}.
		\end{align*}
	For $r$ as in Lemma \ref{Lem: Mainak}, the function $g(z_1, z_2, z_3) = f(rz_1, rz_2, rz_3)$ is holomorphic and contractive on $\mathbb{D}^3$. Then, applying Lemma \ref{Lem: Knese} along with the chain rule, we get
	\begin{align}\label{Eq: Schwarz lemma}
	r	\sum_{j=1}^3  \left| \frac{\partial f}{\partial z_j}(0, 0, 0) \right| \leq 1- |f(0, 0, 0)|^2 =1.
	\end{align}
	
	Since $0< | \lambda_j | < r$ for $j=1, 2, 3$, the inequality \eqref{Eq: Schwarz lemma} yields
	\begin{align}
		\|	f(T_1, T_2, T_3) \| = \left| \sum_{j=1}^3 \lambda_j \frac{\partial f}{\partial z_j}(0, 0, 0) \right|  \leq 1.
	\end{align}
	
	If $f(0, 0, 0) = \alpha \neq 0$, then either $|\alpha|=1$ or $|\alpha|<1$. When $|\alpha | =1$, $f$ is a constant function and in that case we obviously have $\|	f(T_1, T_2, T_3) \| \leq 1$. Now let $|\alpha| < 1$. Then the function $B_\alpha \circ f $ is in $\mathcal{A}(\overline{\mathbb{E}})$ with $\| B_\alpha \circ f\|_{\infty, \overline{\mathbb{E}}} \leq 1$ and  $B_\alpha \circ f(0, 0, 0)=0$., where $B_\alpha$ is a Blaschke factor sending $\alpha$ to $0$.
	Thus, from our previous argument we conclude that 
	$$\|(B_\alpha \circ f)(T_1, T_2, T_3)\| \leq 1.$$
	Applying von-Neumann inequality for the single contraction $(B_\alpha \circ f)(T_1, T_2, T_3)$, we get 
	$$
	\|f(T_1, T_2, T_3)\| = \| B_\alpha^{-1} ((B_\alpha \circ f)(T_1, T_2, T_3))\| \leq 1.
	$$
	Hence, $(T_1, T_2, T_3)$ is a tetrablock contraction for $\lambda_1, \lambda_2, \lambda_3$ as specified above.

Furthermore, we compute the fundamental operators $F_1$ and $F_2$ of $(T_1, T_2, T_3)$. Here,
\begin{align*}
	D_{T_3} = (I- T_3^* T_3)^{1/2} = \begin{bmatrix}
		1 & 0 \\ 0 & (1-|\lambda_3|^2)^{1/2}
	\end{bmatrix}.
\end{align*}
	
	Therefore, 
	\begin{align*}
		 F_1= D_{T_3}^{-1}   (T_1 - T_2^*T_3)  D_{T_3}^{-1} = D_{T_3}^{-1} \begin{bmatrix}
			0 & \lambda_1 \\ 0 & - \bar{\lambda}_2 \lambda_3
		\end{bmatrix} D_{T_3}^{-1} = \begin{bmatrix}
		0 & \frac{\lambda_1}{\sqrt{1-|\lambda_3|^2}} \\ 0 & - \frac{\bar{\lambda}_2 \lambda_3}{1-|\lambda_3|^2}
		\end{bmatrix}
			\end{align*}
		and
			\begin{align*}
		 F_2= D_{T_3}^{-1}   (T_2 - T_1^*T_3)  D_{T_3}^{-1} = D_{T_3}^{-1} \begin{bmatrix}
			0 & \lambda_1 \\ 0 & - \bar{\lambda}_1 \lambda_3
		\end{bmatrix} D_{T_3}^{-1} = \begin{bmatrix}
		0 & \frac{\lambda_2}{\sqrt{1-|\lambda_3|^2}} \\ 0 & - \frac{\bar{\lambda}_1 \lambda_3}{1-|\lambda_3|^2}
		\end{bmatrix}
	\end{align*}
	
	Thus, 
	\begin{align*}
		F_1 F_2 = \begin{bmatrix}
			0 & - \frac{|\lambda_1|^2 \lambda_3}{(1-|\lambda_3|^2)^{3/2}} \\
			0 & \frac{\bar{\lambda}_1 \bar{\lambda}_2 \lambda_3^2}{(1-|\lambda_3|^2)^{2}} 
		\end{bmatrix} \ \text{ and } 
		F_2 F_1 =  \begin{bmatrix}
			0 & - \frac{|\lambda_2|^2 \lambda_3}{(1-|\lambda_3|^2)^{3/2}} \\
			0 & \frac{\bar{\lambda}_1 \bar{\lambda}_2 \lambda_3^2}{(1-|\lambda_3|^2)^{2}} 
		\end{bmatrix}
	\end{align*}
	Since $|\lambda_1| \neq |\lambda_2|$ and $\lambda_3 \neq 0$, we have $F_1 F_2 \neq F_2 F_1$.
		\end{proof}
	
	So, Theorem \ref{Thm: Mainak}, along with Corollary \ref{Cor:Main}, shows that the tetrablock contractions $(T_1, T_2, T_3)$ as described above have tetrablock unitary dilations, but the fundamental operators corresponding to each of them are non-commuting.
	
	\section{Concluding Remarks}
  Although the tetrablock contractions described in Theorem \ref{Thm: Mainak} have tetrablock unitary dilation, we do not have the explicit tetrablock isometries/unitaries that dilate them. The non-commutativity of the fundamental operators implies that the dilation tuple is not of the form \eqref{Eq:Dil-tuple} in Theorem \ref{Thm: B}. 
  
  In \cite{BB}, a more general construction of the dilation tuple has been provided, which encompasses a larger collection of tetrablock contractions by relaxing the condition $[F_1, F_1^*] = [F_2, F_2^*]$ of Theorem \ref{Thm: B}. However, the commutativity of $F_1$ and $F_2$ remains necessary to have such a dilation tuple.
  In this case, the dilation space is not the minimal dilation space of $T_3$. More precisely, in \eqref{Eq:Dil-tuple}, $\varphi_3(z)$ has been taken to be $I_{\mathcal{D}_{T_3}} z^2$ and hence, we must have $\varphi_1(z)= F_1 + \Xi z + F_2^* z^2$ and $\varphi_2(z)= F_2 + \Xi^*z + F_1^*z^2$, for some bounded operator $\Xi$ on $\mathcal{D}_{T_3}$ satisfying the conditions of Theorem 4.1 in \cite{BB}.

 Furthermore, in \eqref{Eq:Dil-tuple}, we can take $\varphi_3(z) =I_{\mathcal{D}_{T_3}} z^n$ and change the position of $D_{T_3}$ in the column operator $C_3$ to any position to obtain a different tetrablock isometric dilation. In such case, a modification of the proof of Theorem 4.1 in \cite{BB} will show that the symbols, $\varphi_1$ and $\varphi_2$ will become $\mathcal{B}(\mathcal{D}_{T_3})$-valued polynomials of degree at most $n$, whose coefficients will satisfy certain relations along with: $$\varphi_1(0)= F_1= \frac{1}{n!} \varphi_2^{(n)}(0)^* \text{ and } \varphi_2(0)= F_2= \frac{1}{n!} \varphi_1^{(n)}(0)^*.$$ Thus, if a tetrablock contraction has such a tetrablock isometric dilation, then the commutativity of $F_1$ and $F_2$ is necessary, as can be observed from the commutativity of $V_1$ and $V_2$.

 Till now, these are the known explicit constructions for the tetrablock isometric dilation of tetrablock contractions satisfying several conditions, including the commutativity of the fundamental operators. Therefore, an explicit construction of the tetrablock isometric dilations for the $2 \times 2$ tetrablock contractions produced in the previous section remains unknown because of non-commuting fundamental operators. Finding dilation tuples for these tetrablock contractions is worth exploring, and it might bring new ideas to tackle the long-standing rational dilation problem for the tetrablock. 
  
  \vspace{3mm}
  
 \noindent  \textbf{Acknowledgement:} The author thanks Prof. Tirthankar Bhattacharyya for carefully reading the initial draft of this note. The author is thankful for financial support from ANRF-NPDF (Grant No. PDF/2025/002836) and from the Indian Institute of Science, Bengaluru (Grant No. SP/DSTO-21-0263.26).
 He also thanks the anonymous referee for their comments.

\end{document}